\documentclass{amsart}
\usepackage{comment}
\usepackage{tikz-cd}
\usepackage{xcolor}
\newbox\xratbelow
\newbox\xratabove
\usepackage{mathabx}
\usepackage{cancel}
\usepackage{float}
\usepackage{geometry}
\newtheorem{theorem}{Theorem}[section]
\newtheorem{corollary}[theorem]{Corollary}

\newtheorem{lemma}[theorem]{Lemma}
\newtheorem{proposition}[theorem]{Proposition}

\theoremstyle{definition}
\newtheorem{definition}[theorem]{Definition}
\newtheorem{remark}[theorem]{Remark}

\newtheorem*{example}{Example}
\newtheorem*{namedcorollary}{Corollary of Lemma \ref{lemma:AK}}

\begin{document}

\title[Anti-integrable limits for almost-periodic systems]{Anti-integrable limits for generalized Frenkel-Kontorova models on almost-periodic media}
\author{Jianxing Du,  Xifeng Su}

\address{School of Mathematical Sciences, Beijing Normal University, No. 19, XinJieKouWai St., HaiDian District, Beijing 100875, P. R. China}
\email{jianxingdu@mail.bnu.edu.cn\\ jxdu000@gmail.com}
\address{School of Mathematical Sciences, Laboratory of Mathematics and Complex Systems (Ministry of Education)\\
Beijing Normal University,
No. 19, XinJieKouWai St., HaiDian District, Beijing 100875, P. R. China}
\email{xfsu@bnu.edu.cn, billy3492@gmail.com}

\date{\today}

\begin{abstract}
We study the equilibrium configurations for  generalized Frenkel-Kontorova models  subjected to almost-periodic media.
By contrast with the spirit of  the KAM theory, our approach consists in establishing the other perturbation theory for fully chaotic systems far away from the integrable, which is called ``anti-integrable" limits.
More precisely, we show that for large enough potentials, there exists a locally unique equilibrium with any prescribed rotation number/vector/plane, which is hyperbolic. The assumptions are general enough to satisfy both short-range and long-range Frenkel-Kontorova models and their multidimensional analogues.
\end{abstract}

\maketitle
\tableofcontents

\section{Introduction}
The classical Frenkel-Kontorova models  describe one dimensional chains of interacting particles subjected to one dimensional periodic media as in \cite{FK39, BK2004}. 
The state of a system is given by $u=\{u_n\}_{n\in\mathbb Z}$ with $u_n\in\mathbb{R}$, and we call it  the configuration of that system. 
The interaction of the particles with the substratum is modeled by a periodic function $V: \mathbb{R}\rightarrow \mathbb{R}$.
The physical states are then selected to be the critical points of the following formal energy functional:
\begin{equation}\label{short}
\mathcal S(\{u_n\}_{n\in\mathbb Z})=\sum_{n\in\mathbb Z} \left[\frac12(u_n-u_{n+1})^2+ \lambda V(u_n)\right] \qquad \text{ for some } \lambda\in \mathbb{R}.
\end{equation}
The critical points of $\mathcal{S}$ are obtained by taking its formal
derivatives and setting them to zero, i.e.,
\begin{equation}\label{equilibrium}
u_{n+1}+u_{n-1}-2u_n- \lambda V'(u_n)=0.
\end{equation}

In this paper, we are concerned with  models subjected to almost-periodic media, that is, we choose $V$ to be an almost-periodic function $V(\theta)=\widehat{V}(\theta\alpha)$, $\widehat V:\mathbb T^{\mathbb N}\to\mathbb R$, with infinite-dimensional rationally independent frequency $\alpha\in[0,1]^{\mathbb N}$. One example to keep in mind would be $V(\theta) = \sum\limits_{n=0}^\infty \frac{1}{2^n} \cos{\frac{1}{\pi^n}\theta}$.\smallskip

As is standard in the literature, there are two perturbation theories for constructing solutions to the equilibrium equation \eqref{equilibrium}: one is the celebrated Kolmogorov-Arnold-Moser(KAM) theory which essentially focuses on perturbations of integrable dynamical systems, and the other is the theory of the anti-integrable limits, which deals with perturbations of fully chaotic dynamical systems. 

On one hand, the KAM theory for equilibria in one dimensional almost-periodic media was recently established in \cite{ADSWY}, while the KAM theory for periodic and quasi-periodic potential could be found in works such as \cite{Rafaelbook, SuL2012, SuL2012b} and the references therein. We point out that these KAM equilibria for periodic potentials can be interpreted as orbits of some twist map and are indeed elliptic.

On the other hand, the theory of anti-integrable limits for periodic potentials was first obtained in \cite{Aubry1990} (see e.g. \cite{Bolotin_2015} for a survey), and the extension of this theory to pattern-equivariant potentials was provided in \cite{Trevino2019}. Note that these equilibria for periodic potentials derived from anti-integrable limits are hyperbolic, see \cite{Goroff, Bolotin_2015}. Moreover, they could be hyperbolic Cantori, which could be thought of as remnants of KAM tori (see \cite{Mackay_1992}).

As for non-perturbative approaches, it is worthwhile to mention the celebrated Aubry-Mather theory developed independently by S. Aubry \cite{Aubry83} and J. Mather \cite{Mather82}, which utilizes variational methods to look for ground states. However, extending the Aubry-Mather theory to quasi-periodic or almost-periodic media presents significant challenges, and we will go in this direction in the present paper.

The aim of this paper is to develop the theory of anti-integrable limits for generalized mechanical systems driven by almost-periodic potentials. Specifically, we make the following contributions:
\begin{itemize}
\item we  establish anti-integrable limits in a  more general setting, see  Theorem \ref{thm:main} in Section~\ref{sec:main} and 

\item we prove the hyperbolicity of the anti-integrable equilibria, see Theorems~\ref{thm:hyperbolic} and ~\ref{thm:DefofHyperbolic} in Section~\ref{sec:hyperbolic}.  
\end{itemize}

As an application of the above results, we present the following theorem for the short-range  model \eqref{short} with an almost-periodic potential. To the best of our knowledge, this result is new.
\begin{theorem}
Let $V$ be an almost-periodic function of class $C^2$ with at least one non-degenerate critical point.    Then, there exists a constant $\lambda_0>0$, determined by $V$, such that for any $\lambda>\lambda_0$ and $\rho\in\mathbb{R}$, there exists a sequence $\{u_n\}_{n\in\mathbb{Z}}$ that satisfies the equilibrium equation \eqref{equilibrium} and fulfills the condition  $$\lim_{n\to\infty}\frac{u_n-u_0}{n}=\rho.$$ Moreover, the sequence $\{(u_n,u_n-u_{n-1})\}_{n\in\mathbb{Z}}$  constitutes a hyperbolic orbit of the map $(x,p)\mapsto(x+p+\lambda V'(x),p+\lambda V'(x))$.
\end{theorem}

The proof strategy is as follows:
\begin{itemize}
\item we first provide a verifiable criterion in Section~\ref{sec:preliminary}, called the Aubry criterion (see Definition~\ref{def:Aubry criterion}), which a modification of the extensions in \cite[Section 5]{Aubry1990} (see also \cite[Section 4]{Bolotin_2015});

\item we then establish a unified framework on anti-integrable limits for functions satisfying the Aubry criterion in Section~\ref{sec:main}, that is, we  show that if the derivative of potential $V$ in \eqref{equilibrium} satisfies the Aubry criterion,  the anti-integrable limits exist for large enough parameter $\lambda$;

\item consequently, we verify  in Section~\ref{sec:EgOfAubry} that periodic functions, pattern-equivariant functions and almost-periodic functions satisfy the Aubry criterion;

\item we finally we prove the existence of a uniformly hyperbolic structure on the orbits constructed from the anti-integrable limits for functions  satisfying the Aubry criterion in Section~\ref{sec:hyperbolic}.
\end{itemize}

%

\section{Preliminary}\label{sec:preliminary}
\subsection{Invariant operators}

Let $d>0$ and $s>0$ be two integers.  Consider the space $\mathcal{C}$, which consists of all functions from $\mathbb{Z}^s$ to $\mathbb{R}^d$. We refer to  the elements in $\mathcal{C}$ as  configurations, and typically write $u_i:=u(i)$ for $i\in\mathbb{Z}^s$ and $u\in\mathcal{C}$. Let $\operatorname{Hom}(\mathbb{Z}^s,\mathbb{R}^d)$ be the subspace of $\mathcal{C}$ containing all homomorphisms in $\mathcal{C}$.

Let $\|\cdot\|$ represent the standard norm on $\mathbb{R}^d$. The interior and boundary of a subset $A\subset\mathbb{R}^d$ are denoted by $\operatorname{Int}(A)$ and $\partial A$, respectively.  The open ball centered at $x$ with radius $r$ in $\mathbb{R}^d$ is denoted by by $B_r(x)$, and  the closed ball is denoted by $\overline{B}_r(x)$.  For simplicity, the symbol  $0$  is used interchangeably to represent the zero natural number, the zero vector in  $\mathbb{R}^d$ , and the zero configuration in  $\mathcal{C}$. 

We define  an extended metric\footnote{An extended metric is a metric so as to allow the distance function $d$ to attain the value $+\infty$.} on $\mathcal{C}$ as follows:\begin{align*}
    d(u,u')=\sup_{i\in\mathbb{Z}^s}\|u_i-u'_i\|.
\end{align*} 

\begin{remark}
  \begin{enumerate}
    \item [(i)] If we focus only on the topological properties of an extended metric space, the extended metric $d$ can be replaced by its bounded counterpart 
    \begin{align*}
      \bar{d}=\min\{d,1\}.
    \end{align*}
    The topology defined by $\bar{d}$ is equivalent to that defined by $d$. However, properties related to boundedness may no longer hold under this analogy.
     \item [(ii)] Another way to transform the extended metric  $d$  into a proper metric is to restrict the distance function to a subspace of  $\mathcal{C}$. This approach will be utilized in the proof of Theorem \ref{thm:main} to ensure the applicability of Banach’s fixed point theorem. 
       \end{enumerate}
\end{remark}

Let $S^k:\mathcal{C}\rightarrow\mathcal{C}$ denote the shift operator along $k\in\mathbb{Z}^s$, defined as $S^k(u)_i=u_{i+k}$ for all $i\in\mathbb{Z}^s$. Similarly, let $T^c:\mathcal{C}\rightarrow\mathcal{C}$ denote the translation operator along the vector $c\in\mathbb{R}^d$, defined as $T^c(u)_i=u_i+c$ for all $i\in\mathbb{Z}^s$. 

\begin{definition}[Invariant operator]
  \begin{enumerate}
    \item [(i)] An operator $\Delta:\,\mathcal{C}\rightarrow\mathcal{C}$ is called \emph{horizontally invariant} if it commutes  with the shift operator, i.e., \begin{align*}
      \Delta\circ S^k(u)=S^k\circ \Delta(u), \quad \forall u\in\mathcal{C}, k\in\mathbb{Z}^s.
  \end{align*}
  \item [(ii)] An operator $\Delta:\,\mathcal{C}\rightarrow\mathcal{C}$ is called \emph{vertically invariant} if it is invariant under translations, i.e., \begin{align*}
    \Delta\circ T^c(u)=\Delta(u), \quad \forall u\in\mathcal{C}, c\in\mathbb{R}^d.
\end{align*}
  \item [(iii)] An operator $\Delta:\,\mathcal{C}\rightarrow\mathcal{C}$ is called invariant if it is both horizontally and vertically invariant.
  \end{enumerate}
\end{definition}

\begin{lemma}
    If $\Delta$ is an invariant operator, then for any $\rho\in\operatorname{Hom}(\mathbb{Z}^s,\mathbb{R}^d)$, $\Delta(\rho)$ is a constant sequence. That is, $\Delta(\rho)_k=\Delta(\rho)_{0}$ for all $k\in\mathbb{Z}^s$.
\end{lemma}

\begin{proof}
    For any $k\in\mathbb{Z}^s$, since $\Delta$ is invariant, we have:\begin{align*}
        \Delta(\rho)=\Delta\circ T^{\rho(k)}(\rho)=\Delta\circ S^k(\rho)=S^k\circ\Delta (\rho).
    \end{align*}
    Taking the value at $0$, we obtain \begin{align*}
        \Delta(\rho)_0=\left[S^k\circ\Delta (\rho)\right]_0=\Delta(\rho)_k.
    \end{align*}
\end{proof}
Using the proceeding lemma, we can define a function \begin{align*}
    \left.\Delta\right|_{\operatorname{Hom}}:\,\operatorname{Hom}(\mathbb{Z}^s,\mathbb{R}^d)\rightarrow\mathbb{R}^d\\ 
    \rho\mapsto \Delta(\rho)_0
\end{align*}

\begin{definition}[Lipschitz continuous operator]
  Let $\mathcal{C}'\subset\mathcal{C}$ be a subset of $\mathcal{C}$. An operator $\Delta:\,\mathcal{C}\rightarrow\mathcal{C}$ is said to be  \emph{Lipschitz continuous} on $\mathcal{C}'$ if there exists $K>0$  such that $d(\Delta u,\Delta u')\leq K d(u,u')$ for all $u,u'\in\mathcal{C}'$.
\end{definition}

An important case is when $\Delta$ is Lipschitz continuous on the closed ball $\{u\mid d(u,\rho)\leq R\}$, where $\rho\in\operatorname{Hom}(\mathbb{Z}^s,\mathbb{R}^d)$ and $R>0$.  The Lipschitz constant on this closed ball is denoted by $K(\rho,R)$.

\begin{example}[Long-range Frenkel-Kontorova models]
  Let $s=d=1$. Consider the long range cubic interaction given by: \begin{align*}
    \Delta(u)_i=\sum_{k\in\mathbb{Z}}2^{-|k|}(u_i-u_{i+k})^3,\quad \forall i\in\mathbb{Z}.
  \end{align*}
  It is easy to verify that $\Delta$ is invariant. To show that  $\Delta$ is Lipschitz continuous on  $\{u\mid d(u,\rho)\leq R\}$, observe that: \begin{align*}
    &\quad\,\|\Delta(u)_i-\Delta(u')_i\|\\ 
    &\leq  \sum_{k\in\mathbb{Z}}2^{-|k|}\|(u_i-u_{i+k})^3-(u_i'-u_{i+k}')^3\|\\ 
    &\leq \sum_{k\in\mathbb{Z}}2^{-|k|}2d(u,u')(\|u_i-u_{i+k}\|^2+\|u_i-u_{i+k}\|\cdot\|u_i'-u_{i+k}'\|
    +\|u_i'-u_{i+k}'\|^2)\\ 
    &\leq \sum_{k\in\mathbb{Z}}2^{-|k|}2d(u,u')3(k|\rho|+2R)^2\\ 
    &=\left[\sum_{k\in\mathbb{Z}}\frac{6(k|\rho|+2R)^2}{2^{|k|}}\right]d(u,u')
  \end{align*}
  for all $u,u'\in \{u\mid d(u,\rho)\leq R\}$ and $i\in\mathbb{Z}$. Thus, $\Delta$ is Lipschitz continuous on  $\{u\mid d(u,\rho)\leq R\}$ with $K(\rho, R)=\sum_{k\in\mathbb{Z}}\frac{6(k|\rho|+2R)^2}{2^{|k|}}$.
\end{example}

\subsection{Aubry criterion}

\begin{definition}[Pointwise operator]
    An operator $\Psi:\,\mathcal{C}\rightarrow\mathcal{C}$ is called a  \emph{pointwise operator} if there exists a function $\psi:\mathbb{R}^d\rightarrow\mathbb{R}^d$ such that  \begin{align*}
        \Psi(u)_i=\psi(u_i)
    \end{align*}
    for all $i\in\mathbb{Z}^s$  and $u\in\mathcal{C}$. The function $\psi$ is referred to as the \emph{pointwise function} of the operator  $\Psi$.
\end{definition}

\begin{definition}[Aubry criterion]\label{def:Aubry criterion}
  A continuous function $\psi:\mathbb{R}^d\rightarrow\mathbb{R}^d$ is said to satisfy \emph{Aubry criterion} if  there exists a subset $O\subset \mathbb{R}^d$ and constants $R>0$, $r>0$ and $m>0$  such that: \begin{enumerate}
          \item [(i)]  every ball with radius  $R$ (whatever be its center)  contains at least one point in $O$;
          \item [(ii)] $\psi(z)=0,\,\forall z\in O$;
          \item [(iii)] for any $z\in O$ and $x,y\in\overline{B}_r(z)$, $\|\psi(x)-\psi(y)\|\geq m\|x-y\|$.
      \end{enumerate}
\end{definition}

\begin{remark}
\begin{enumerate}
	\item [(i)]The concept of the Aubry criterion was first introduced in \cite{Aubry1990}. Several other works, such as \cite{Mackay_1992}, \cite{Trevino2019}, and \cite{Bolotin_2015}, have explored similar or generalized notions. Our criterion encompasses the class of functions described in these papers. Figure 1 illustrates a typical example of a one-dimensional function that satisfies the Aubry criterion. For additional examples, refer to Section \ref{sec:EgOfAubry}.
	\item [(ii)] If $\psi$ satisfies the Aubry criterion with parameters $O, R,r$ and $m$, then $\psi$ admits a $1/m$-Lipschitz local inverse near any point of $O$ (see Proposition \ref{prop:Lipschitz}).
\end{enumerate}
\end{remark}

\begin{figure}[H]
  \centering

\tikzset{every picture/.style={line width=0.75pt}} 

\begin{tikzpicture}[x=0.75pt,y=0.75pt,yscale=-0.75,xscale=0.75]

\draw [line width=1.5]    (193.09,148.06) .. controls (206.09,127.06) and (200.09,109.06) .. (215.09,97.06) .. controls (230.09,85.06) and (218.01,72.07) .. (244.01,55.07) ;
\draw [line width=1.5]    (332.01,120.16) .. controls (353.01,117.16) and (346.58,111.05) .. (361.79,97.61) .. controls (377.01,84.16) and (361.01,59.16) .. (387.01,42.16) ;
\draw [line width=1.5]    (77.01,53.16) .. controls (96.01,84.16) and (76.17,86.97) .. (104.09,98.06) .. controls (132.01,109.16) and (117.01,120.16) .. (132.01,127.16) ;
\draw  (17.17,98.18) -- (649.17,98.18)(334.49,28) -- (334.49,148.32) (642.17,93.18) -- (649.17,98.18) -- (642.17,103.18) (329.49,35) -- (334.49,28) -- (339.49,35)  ;
\draw [color={rgb, 255:red, 208; green, 2; blue, 27 }  ,draw opacity=1 ][line width=3]    (75.58,98.89) -- (134.01,98.32) ;
\draw  [color={rgb, 255:red, 74; green, 99; blue, 226 }  ,draw opacity=1 ][fill={rgb, 255:red, 74; green, 144; blue, 226 }  ,fill opacity=1 ] (100,98.06) .. controls (100,95.8) and (101.83,93.98) .. (104.09,93.98) .. controls (106.34,93.98) and (108.17,95.8) .. (108.17,98.06) .. controls (108.17,100.32) and (106.34,102.15) .. (104.09,102.15) .. controls (101.83,102.15) and (100,100.32) .. (100,98.06) -- cycle ;
\draw [color={rgb, 255:red, 208; green, 2; blue, 27 }  ,draw opacity=1 ][line width=3]    (186.58,97.89) -- (245.01,97.32) ;
\draw  [color={rgb, 255:red, 74; green, 99; blue, 226 }  ,draw opacity=1 ][fill={rgb, 255:red, 74; green, 144; blue, 226 }  ,fill opacity=1 ] (211,97.06) .. controls (211,94.8) and (212.83,92.98) .. (215.09,92.98) .. controls (217.34,92.98) and (219.17,94.8) .. (219.17,97.06) .. controls (219.17,99.32) and (217.34,101.15) .. (215.09,101.15) .. controls (212.83,101.15) and (211,99.32) .. (211,97.06) -- cycle ;

\draw  [color={rgb, 255:red, 74; green, 99; blue, 226 }  ,draw opacity=1 ][fill={rgb, 255:red, 74; green, 144; blue, 226 }  ,fill opacity=1 ] (531,39.06) .. controls (531,36.8) and (532.83,34.98) .. (535.09,34.98) .. controls (537.34,34.98) and (539.17,36.8) .. (539.17,39.06) .. controls (539.17,41.32) and (537.34,43.15) .. (535.09,43.15) .. controls (532.83,43.15) and (531,41.32) .. (531,39.06) -- cycle ;
\draw [color={rgb, 255:red, 208; green, 2; blue, 27 }  ,draw opacity=1 ][line width=3]    (481.58,63.89) -- (540.01,63.32) ;
\draw [color={rgb, 255:red, 208; green, 2; blue, 27 }  ,draw opacity=1 ][line width=3]    (332.58,97.89) -- (391.01,97.32) ;
\draw  [color={rgb, 255:red, 74; green, 99; blue, 226 }  ,draw opacity=1 ][fill={rgb, 255:red, 74; green, 144; blue, 226 }  ,fill opacity=1 ] (357,97.06) .. controls (357,94.8) and (358.83,92.98) .. (361.09,92.98) .. controls (363.34,92.98) and (365.17,94.8) .. (365.17,97.06) .. controls (365.17,99.32) and (363.34,101.15) .. (361.09,101.15) .. controls (358.83,101.15) and (357,99.32) .. (357,97.06) -- cycle ;

\draw [line width=1.5]  [dash pattern={on 1.69pt off 2.76pt}]  (132.01,127.16) .. controls (142.01,133.32) and (141.01,42.32) .. (148.01,42.32) .. controls (155.01,42.32) and (146.01,134.32) .. (164.01,115.32) .. controls (182.01,96.32) and (168.09,173.06) .. (193.09,148.06) ;
\draw [line width=1.5]  [dash pattern={on 1.69pt off 2.76pt}]  (24.01,45.32) .. controls (27.01,26.32) and (67.01,28.32) .. (77.01,53.16) ;
\draw [line width=1.5]  [dash pattern={on 1.69pt off 2.76pt}]  (244.01,55.07) .. controls (266.01,44.32) and (261.01,151.32) .. (266.01,152.32) .. controls (271.01,153.32) and (274.01,22.32) .. (276.01,21.32) .. controls (278.01,20.32) and (281.01,140.32) .. (289.01,137.32) .. controls (297.01,134.32) and (293.01,29.32) .. (299.01,26.32) .. controls (305.01,23.32) and (297.01,153.32) .. (312.01,146.32) .. controls (327.01,139.32) and (317.01,123.32) .. (332.01,120.16) ;
\draw [line width=1.5]  [dash pattern={on 1.69pt off 2.76pt}]  (387.01,42.16) .. controls (409.01,26.32) and (420.01,87.32) .. (461.01,56.32) ;

\draw (548,29) node [anchor=north west][inner sep=0.75pt]   [align=left] {$\displaystyle O$};
\draw (548,55) node [anchor=north west][inner sep=0.75pt]   [align=left] {$\displaystyle O+[ -r,r]$};

\end{tikzpicture}

\caption{The partial graph of a one-dimensional function satisfying the Aubry criterion is illustrated. The portion within $O+[-r,r]$ is  represented by a solid line, while the portion outside $O+[-r,r]$ is depicted with a dotted line.}
\end{figure}
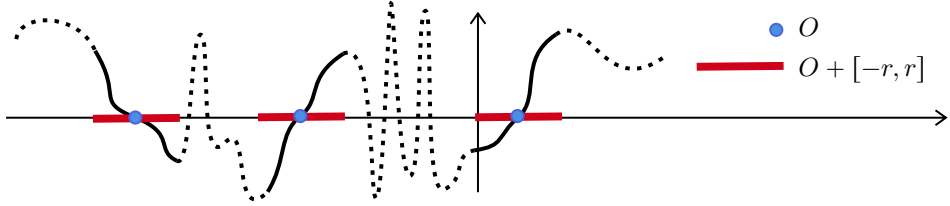

\begin{proposition}\label{prop:Lipschitz}
    Let $\psi$ be a continuous function satisfying the Aubry criterion with parameters $O,R,r,m$. For any $z\in O$, there exists a function $\varphi_z:\overline{B}_{rm}(0)\rightarrow \overline{B}_r(z)$ such that \begin{enumerate}
      \item [(i)] $\psi\circ \varphi_z(x)=x,$ for all $x\in \overline{B}_{rm}(0).$
      \item [(ii)] $ \|\varphi_z(x)-\varphi_z(y)\|\leq \frac{1}{m}\|x-y\|$, for all $x,y\in \overline{B}_{rm}(0)$.
    \end{enumerate}
\end{proposition}

Before proving Proposition \ref{prop:Lipschitz}, we first establish the following topological lemma.

\begin{lemma}\label{lemma:PointSetTop}
  Let $K$ be a compact subset of $\mathbb{R}^d$. If $f:K\rightarrow\mathbb{R}^d$ is continuous and injective, then \begin{enumerate}
  	\item [(i)]$f(\operatorname{Int}(K))\subset \operatorname{Int}[f(K)]$;
  	\item [(ii)] $\partial [f(K)]\subset f(\partial K)$. 
  \end{enumerate}
  \end{lemma}

\begin{proof}
  \begin{enumerate}
    \item [(i)] Let $y\in f(\operatorname{Int}(K))$. Then there exists a unique $x\in\operatorname{Int}(K)$ such that $f(x)=y$. Since $x\in\operatorname{Int}(K)$, there exists a ball $B_r(x)\subset \operatorname{Int}(K)$. By Brouwer's invariance of domain, $f(B_r(x))\subset f(K)$ is open. Since $y=f(x)\in f(B_r(x))$, we conclude that $y\in  \operatorname{Int}[f(K)]$.
    \item [(ii)] Let $y\in \partial [f(K)]$.  Then there exists a sequence $\{y_k\}_{k\in\mathbb{N}}\subset f(K)$ such that $y_k\rightarrow y$. Since $f$ is injective, we can define the sequence $\{x_k:=f^{-1}(y_k)\}_{k\in\mathbb{N}}\subset K$. As $K$ is compact, there exists a subsequence $\{x_{k_n}\}_{n\in\mathbb{N}}$ converges to some $x\in K$. By the continuity of $f$, we have $f(x)=y$. Next, we show that $x\in\partial K$. Since $K$ is compact, it is closed, so $K=\operatorname{Int}(K)\cup \partial K$. If $x\in \operatorname{Int}(K)$, then $y\in f(\operatorname{Int}(K))\subset \operatorname{Int}[f(K)]$ (from (i)), which contracts the fact that $y\in \partial [f(K)]$. Therefore, $x\not\in\operatorname{Int}(K)$, implying $x\in \partial K$.
  \end{enumerate}
  This concludes the proof.
\end{proof}

\begin{proof}[Proof of Proposition \ref{prop:Lipschitz}]
  Let $z\in O$ and $x,y\in \overline{B}_r(z)$. Since $\psi$ satisfies the Aubry criterion, we have
  \begin{align}\label{eq:InverseLipschitz}
    \|\psi(y)-\psi(x)\|\geq m\|y-x\|.
  \end{align}
    Thus, $\psi(x)=\psi(y)$ if and only if $x=y$. That is, $\psi$ is injective on $\overline{B}_r(z)$. 
    Since $\psi$ is continuous and injective, by the Jordan-Brouwer separation theorem, the image $S:=\psi(\partial B_r(z))$ of the boundary sphere  divides $\mathbb{R}^d\setminus S$ into two connected components $V_1$ and $V_2$, with $S=\partial V_1=\partial V_2$.  Let $V_1$ denote the bounded component. We claim that $V_1=\psi(B_r(z))$: \begin{enumerate}
    	\item [(a)]on  one hand, $\psi(B_r(z))$ is open in $\mathbb{R}^d\setminus S$ by Brouwer's invariance of domain;
    	\item [(b)]on the other hand, $\psi(B_r(z))$ is closed in $\mathbb{R}^d-S$ since  its boundary $\partial \left[\psi(B_r(z))\right]\subset S$ by Lemma \ref{lemma:PointSetTop}.
    \end{enumerate}   
    Thus, by the boundedness and connectedness of $V_1$, we conclude that $V_1=\psi(B_r(z))$. Next, by the  inequality \eqref{eq:InverseLipschitz}, we have: \begin{align*}
      S\cap  B_{rm}(0)=\emptyset.
    \end{align*}
    Indeed, suppose  there exists $p\in S\cap  B_{rm}(0)$. Then: \begin{align*}
      mr=m\|\psi^{-1}(p)-z\|\leq \|p-\psi(z)\|=\|p\|<mr,
    \end{align*}
    which is a contradiction.
    Thus $B_{rm}(0)$ is contained entirely in either $V_1$ or $V_2$.
    Since $0\in B_{rm}(0)$ and $ 0=\psi(z)\in V_1$, we conclude $B_{rm}(0)\subset V_1$. Hence, $\overline{B}_{rm}(0)\subset\overline{V_1}=V_1\cup S=\psi(\overline{B}_r(z))$.
    Since $\psi$ is injective on $\overline{B}_r(z)$ and $\overline{B}_{rm}(0)\subset \psi(\overline{B}_r(z))$, we can define the local inverse of $\psi$:\begin{align*}
        \varphi_z: \overline{B}_{rm}(0)\rightarrow \overline{B}_r(z)
    \end{align*}
    For all $X,Y\in \overline{B}_{rm}(0)$, their preimages $\varphi_z(X),\varphi_z(Y)\in \overline{B}_r(z)$. By the inequality \eqref{eq:InverseLipschitz}, we have: \begin{align*}\label{eq:LipschitzVarphi}
        \|\varphi_z(X)-\varphi_z(Y)\|\leq \frac{1}{m}\|X-Y\|.
    \end{align*}
    This completes the proof of Proposition \ref{prop:Lipschitz}.
\end{proof}

\section{Main theorem and its proof}\label{sec:main}
\begin{theorem}\label{thm:main}
    Let $\Delta$ be an invariant operator, and let $\Psi$ be a pointwise operator with a pointwise function $\psi$. Suppose the following conditions hold:\begin{enumerate}
        \item [(i)] $\Delta$ is Lipschitz continuous on any closed ball  $\{u\mid d(u,\rho)\leq R\}$, with a Lipschitz constant $K(\rho,R)$, where $\rho\in\operatorname{Hom}(\mathbb{Z}^s,\mathbb{R}^d)$;
        \item [(ii)] $\psi$ satisfies the Aubry criterion with parameters $O,R,r,m$.
    \end{enumerate}
    Then, for all $\rho\in \operatorname{Hom}(\mathbb{Z}^s,\mathbb{R}^d)$ and $\lambda\geq \frac{1}{rm}\left(K(\rho,r+R)\cdot(r+R)+\|\left.\Delta\right|_{\operatorname{Hom}}(\rho)\|\right)$, there exists a configuration $u$ such that: \begin{align*}
      (\Delta+\lambda \Psi)u=0,  \quad\sup_{i\in\mathbb{Z}^s}\inf_{z\in O}\|u_i-z\|\leq r\quad\text{ and }\quad  d(u,\rho)\leq r+R.
  \end{align*}
  Moreover, the configuration $u$ is unique in the sense that, if $(\Delta+\lambda \Psi)u'=0$ and for all $i\in\mathbb{Z}^s$ there exists $x\in O$ such that $u_i,u_i'\in \overline{B}_r(x)$, then $u=u'$.
\end{theorem}

\begin{proof}
  Let $\rho\in \operatorname{Hom}(\mathbb{Z}^s,\mathbb{R}^d)$. Define a configuration $a:=a(\rho)\in\mathcal{C}$ by
  \begin{align}
    a_i:=\operatorname{argmin}_{x\in O}\left\|x-\rho (i)\right\|
    \end{align} 
    for all $i\in\mathbb{Z}^s$.
From the definition of the Aubry criterion, we have: \begin{align*}
  d(a,\rho)\leq R.
\end{align*}
Next, consider the space \begin{align*}
  \Pi:=\left\{u:d(u,a)\leq r\right\}.
\end{align*}
We equip the space $\Pi$ with the metric $d$. Note that while $d$ is an extended metric on $\mathcal{C}$, its restriction on $\Pi$ is a proper metric. Referring to Theorem 43.5 in \cite{munkres2000topology}, the metric space $(\Pi,d)$ is complete. The following properties hold: $$\forall u\in\Pi,\quad \sup_{i\in\mathbb{Z}^s}\inf_{z\in O}\|u_i-z\|\leq r\quad\text{ and }\quad  d(u,\rho)\leq r+R,$$ which follow directly from the inequalities $d(u,a)\leq r$ and $d(a,\rho)\leq R$. Consequently, $\Delta$ is Lipschitz continuous on $\Pi$ with Lipschitz constant $K(\rho,r+R)$.
We now claim that the operator \begin{align*}
      \Phi: \Pi \rightarrow \Pi, \quad u \mapsto\left(\varphi_{a_i}\left(\frac{\Delta(u)_i}{\lambda}\right)\right)_{i \in \mathbb{Z}^s}
\end{align*}
where $\varphi_{a_i}$ is the local inverse of $\psi$, is both well-defined and contractive.

  For any $u\in \Pi$ and $i\in\mathbb{Z}^s$, we have  \begin{align*}
      &\|\Delta(u)_i\|\\ 
      =&\|(\Delta(u)_i-\Delta(a)_i)+(\Delta(a)_i-\Delta(\rho)_i)+\Delta(\rho)_i\|\\
      \leq& K(\rho,r+R)d(u,a)+ K(\rho,r+R)d(a,\rho)+\|\left.\Delta\right|_{\operatorname{Hom}}(\rho)\|\\
      \leq&K(\rho,r+R)\cdot (r+R)+\|\left.\Delta\right|_{\operatorname{Hom}}(\rho)\|\\ 
      \leq&rm\lambda.
  \end{align*}
  Thus, \begin{align*}
      \frac{\Delta(u)_i}{\lambda}\in \overline{B}_{rm}(0).
  \end{align*}

  By Proposition \ref{prop:Lipschitz}, for all $i\in\mathbb{Z}^s$ and $u\in\Pi$, we have \begin{align*}
      \left\|\Phi(u)_i-a_i\right\|&=\left\|\varphi_{a_i}\left(\frac{\Delta(u)_i}{\lambda}\right)-\varphi_{a_i}(0)\right\|\\ 
      &\leq \frac{1}{m}\left\|\frac{\Delta(u)_i}{\lambda}\right\|\\ 
      &\leq \frac{1}{m}rm=r.
  \end{align*}
  Hence $\Phi$ is well-defined.

  Again by Proposition \ref{prop:Lipschitz}, for all $i\in\mathbb{Z}^s$ and $u,u'\in \Pi$, we have \begin{align*}
      \left\|\Phi(u)_i-\Phi(u')_i\right\|&=\left\|\varphi_{a_i}\left(\frac{\Delta(u)_i}{\lambda}\right)-\varphi_{a_i}\left(\frac{\Delta(u')_i}{\lambda}\right)\right\|\\
      &\leq \frac{1}{m}\left\|\frac{\Delta(u)_i-\Delta(u')_i}{\lambda}\right\|\\ 
      &\leq \frac{K(\rho,r+R)}{m\lambda}d(u,u')\\ 
      &\leq \frac{r}{r+R}d(u,u').
  \end{align*}
  Therefore, $\Phi$ is a contraction. Then by the Banach's Fixed Point Theorem, $\Phi$ admits a unique fixed point $u\in\Pi$. 	For all $i\in\mathbb{Z}^s$, we then have  \begin{align*}
    \Psi(u)_i=\psi(u_i)=\psi\circ \varphi_{a_i}\left(\frac{\Delta(u)_i}{\lambda}\right)=\frac{\Delta(u)_i}{\lambda}.
\end{align*}
That implies, \begin{align*}
  (\Delta+\lambda \Psi)u=0.
\end{align*}
If, for all $i\in\mathbb{Z}^s$, there exists $x\in O$ such that $u_i,u_i'\in \overline{B}_r(x)$, it follows that  $u,u'\in \Pi$. Moreover,  if both $u$ and $u'$ satisfy the equation $(\Delta+\lambda \Psi)u=0$, then $u$ and $u'$ are the fixed points of $\Phi$. By the uniqueness of fixed points in contraction mappings, we conclude that $u=u'$.
\end{proof}

\section{Examples of functions satisfying the Aubry criterion}\label{sec:EgOfAubry}

\subsection{Periodic functions}

Let $f:\mathbb{R}^d\rightarrow \mathbb{R}$ be a function. The set\begin{align*}
  \operatorname{per}(f):=\{t\in\mathbb{R}^d\mid f(x+t)=f(x),\,\forall x\in\mathbb{R}^d\}
\end{align*}
is called the \emph{set of periods} of $f$. It always contains the trivial period $0$ and forms a subgroup of $(\mathbb{R}^d,+)$. The structure of the group $\operatorname{per}(f)$ determines the type of periodicity of the function $f$.

A point set $\Gamma\subset\mathbb{R}^d$ is called a  \emph{lattice} in $\mathbb{R}^d$ if there exist $d$ linearly independent vectors $b_1,\dots , b_d$ such that
\begin{align*}
  \Gamma=\mathbb{Z} b_1 \oplus \cdots \oplus \mathbb{Z} b_d:=\left\{\sum_{i=1}^d m_i b_i \mid \text { all } m_i \in \mathbb{Z}\right\}
\end{align*}
and  the  $\mathbb{R}$-span of $\Gamma$ is $\mathbb{R}^d$. The set $\{b_1,\dots , b_n\}$ is called a basis of the lattice $\Gamma$.

\begin{theorem}\label{thm:PeriodicAubry}
    Let $f$ be a $C^2$ function with at least one non-degenerate critical point. If $\operatorname{per}(f)$ contains a lattice $\Gamma$ in $\mathbb{R}^d$, then $\nabla f$ satisfies the Aubry criterion.
\end{theorem}

\begin{proof}
Let $z\in\mathbb{R}^d$ be a non-degenerate critical point of $f$. That is, $\nabla f(z)=0$ and $\det(\nabla^2 f(z))\ne 0$. Since $\nabla^2 f(z)$ is symmetric and invertible, its square $(\nabla^2 f(z))^2$ is positive definite. 
As $f$ is $C^2$,  there exist constants $r>0$ and $m>0$ such that, for any $x,y\in\overline{B}_r(z)$, we have $$\xi^T\nabla^2 f(x)\nabla^2 f(y)\xi\geq m^2 \|\xi\|^2,\quad \forall \xi\in\mathbb{R}^d.$$
So for any $x,y\in \overline{B}_r(z)$, we have \begin{align*}
    &\|\nabla f(x)-\nabla f(y)\|^2\\ 
    =&\int_{0}^1\nabla^2f(y+t(x-y))(x-y) \ dt\cdot \int_{0}^1\nabla^2f(y+s(x-y))(x-y) \ ds\\ 
    =&\int_0^1\int_0^1 (x-y)^T \nabla^2f(y+t(x-y))\nabla^2f(y+s(x-y))(x-y)  \ dt ds\\ 
    \geq& \int_0^1\int_0^1 m^2\|x-y\|^2 \ dt ds\\ 
    =&m^2\|x-y\|^2,
\end{align*}
which simplifies to $\|\nabla f(x)-\nabla f(y)\|\geq m\|x-y\|$.
Suppose $b_1,\dots,b_d$ form a basis of $\Gamma$. For any $x\in\mathbb{R}^d$, there exists a unique $(\alpha_1,\dots, \alpha_d)\in\mathbb{R}^d$ such that $x=\sum_{i=1}^d\alpha_ib_i$. Since $\sum_i\lfloor\alpha_i\rfloor b_i\in\Gamma$ and \begin{align*}
    \left\|x-\sum_i\lfloor\alpha_i\rfloor b_i\right\|< \sum_i\|b_i\|,
\end{align*}
for any $x\in\mathbb{R}^d$, the ball with radius $\sum_i\|b_i\|$  contains a point in $\Gamma$.

Let $O:=\Gamma+z$ and $R:=\sum_{i=1}^d\|b_i\|$. Let $r>0$ and $m>0$ be as defined above. Since $\Gamma\subset\operatorname{per}(f)$, the properties of $f$ at any point in $O$ are identical to those at $z$.
It follows that $\nabla f$ satisfies the Aubry criterion with parameters $O,R,r,m$.
\end{proof}

\subsection{Pattern-equivariant functions}

A cluster of the point set $\Lambda\subset\mathbb{R}^d$ is the intersection $K\cap \Lambda$ for some compact set $K\subset\mathbb{R}^d$.  Two clusters $P_1$ and $P_2$  are said to be equivalent if there exists a vector $v\in\mathbb{R}^d$ such that $P_1+v=P_2$.
\begin{definition}[finite local complexity]
A point set $\Lambda\subset \mathbb{R}^d$ is of  finite local complexity if for any $R>0$, the point set $\Lambda$ possesses only finitely many equivalence classes of clusters with diameters smaller than $R$, where the diameter of a subset $A\subset \mathbb{R}^d$ is $\sup_{x,y\in A}|x-y|.$ 
\end{definition}

\begin{definition}[repetitive]
	A point set $\Lambda\subset \mathbb{R}^d$ is repetitive if for any cluster $P\subset \Lambda$, there exists $R>0$ such that  any ball with radius $R$ contains a cluster equivalent to $P$.
\end{definition}

Let $\Lambda\subset \mathbb{R}^d$ be a repetitive point set with finite local complexity (see \cite{GGP2006}\cite{GPT2017}\cite{Trevino2019}, and \cite{dusu2021} for details).   A continuous  function $f:\mathbb{R}^d\rightarrow\mathbb{R}$ is said to be $\Lambda$-equivariant if there exists $R>0$ such that whenever $$(\Lambda-x)\cap B_R(0)=(\Lambda-y)\cap B_R(0),$$
it follows that $f(x)=f(y)$.

\begin{theorem}
    Let $f$ be a $C^2$ function with at least one $non$-degenerate critical point. If $f$ is $\Lambda$-equivariant, then $\nabla f$ satisfies the Aubry criterion.
\end{theorem}

\begin{proof}
    Since $f$ is $\Lambda$-equivariant, there exists $R>0$ such that $$(\Lambda-x)\cap B_R(0)=(\Lambda-y)\cap B_R(0)\quad \Rightarrow\quad  f(x)=f(y).$$ As $\Lambda$ has finite local complexity, the point set $\Lambda$ contains only finitely many equivalence classes of clusters with diameters smaller than $4R$. By repetitivity of $\Lambda$,  we can find $R'>0$ such that any ball with radius $R'$ contains  all clusters with diameters smaller than $4R$. Let $z\in\mathbb{R}^d$ be a non-degenerate critical point of $f$. Then  any ball with radius $R'$ contains the cluster $\Lambda\cap B_{2R}(z)$. In other words, for any ball of radius $R'$, there exists $x$ in the ball such that $$(\Lambda-x)\cap B_{2R}(0)=(\Lambda-z)\cap B_{2R}(0).$$
Since $f$ is $\Lambda$-equivariant, in each ball with radius $R'$, there exists at least one point $s$, such that $$f(s+v)=f(z+v), \forall v\in B_R(0).$$ Let $O$ be the collection of all such points $s$. Now, we can replicate the proof of Theorem \ref{thm:PeriodicAubry}. Specifically, there exist constants $r>0$ and $m>0$ such that  for any $s\in O$ and $x,y\in\overline{B}_r(s)$, we have $\|\nabla f(x)-\nabla f(y)\|\geq m\|x-y\|$. Hence $\nabla f$ satisfies the Aubry criterion with parameters $O,R',r,m$.
\end{proof}

\subsection{Almost-periodic functions}

\begin{definition}
	A point set $\Lambda\subset\mathbb{R}^d$ is called relatively dense if there exists $R>0$ such that  each ball with
radius $R$ (whatever be its center) contains at least a point in $\Lambda$. 
\end{definition}

\begin{definition}
    A $C^k$, function $f:\mathbb{R}^d\rightarrow\mathbb{R}$ is called almost-periodic of class $C^k$, $k=0,1,\dots,\omega$, if for any $\varepsilon>0$, there exists  a relatively dense subset $\Lambda_\varepsilon\subset\mathbb{R}^d$ such that $$\|f(\cdot+x)-f(\cdot)\|_{C^k(\mathbb{R}^d)}<\varepsilon,\quad  \forall x\in \Lambda_\varepsilon.$$
\end{definition}

\begin{theorem}\label{thm:APisAubry}
    Let $f$ be a $C^2$ function with at least one non-degenerate critical point. If $f$ is almost-periodic of class $C^2$, then $\nabla f$ satisfies the Aubry criterion. 
\end{theorem}

\begin{proof}
	Let $z\in\mathbb{R}^d$ be a non-degenerate critical point of $f$, i.e.,  $\nabla f(z)=0$ and $\det(\nabla^2 f(z))\ne 0$. A classic result \cite[Proposition 2.2]{golubitsky1974stable} states that there exists $r>0$ and $\varepsilon>0$ such that $$\|g-f\|_{C^2(\overline{B}_r(z))}<\varepsilon\quad \Rightarrow\quad g=\varphi_g\circ f\circ\psi_g \text{ in }\overline{B}_r(z),$$	
	where $\varphi_g$ and $\psi_g$ are $C^2$-diffeomorphisms. Let $\sigma_{\min}(A)$ be the smallest singular value of a matrix $A$.  Assume that $\varepsilon< \sigma_{\min}(\nabla^2 f(z))/4$ and $r$ is small enough so that $$\sigma_{\min}(\nabla^2 f(x))>\sigma_{\min}(\nabla^2 f(z))/2, \quad \forall x\in \overline{B}_{2r}(z).$$ Then, the following hold:\begin{align*}
		&\|g-f\|_{C^2(\overline{B}_{2r}(z))}<\varepsilon\\ 
		\quad \Rightarrow\quad &\|g-f\|_{C^2(\overline{B}_{r}(z))}<\varepsilon\\ 
		\quad \Rightarrow\quad & g=\varphi_g\circ f\circ\psi_g \text{ in }\overline{B}_r(z)\\ 
		\quad \Rightarrow\quad & \exists z_g\in\overline{B}_{r}(z) \text{ s.t. } \nabla g(z_g)=0. 
	\end{align*}
	In fact, $z_g=(\nabla\psi_g)^{-1}(z)$. Furthermore, 
	 $$\|g-f\|_{C^2(\overline{B}_{2r}(z))}<\varepsilon\quad \Rightarrow\quad \sigma_{\min}(\nabla^2 g(x))\geq \sigma_{\min}(\nabla^2 f(z))/4>\varepsilon,\quad \forall x\in  \overline{B}_{2r}(z).$$
	By replicating the proof of Theorem \ref{thm:PeriodicAubry}, there exist constants $m>0$ and $r'>0$ (independent of $g$) such that,\begin{align*}
		&\|g-f\|_{C^2(\overline{B}_{2r}(z))}<\varepsilon\\ \quad \Rightarrow\quad &\exists z_g\in\overline{B}_r(z), \text{ s.t. } \nabla g(z_g)=0, \forall x,y\in \overline{B}_{r'}(z_g), \|\nabla g(x)-\nabla g(y)\|\geq m\|x-y\|.
	\end{align*}
	Since $f$ is almost-periodic of class $C^2$, there exists a relatively dense subset $\Lambda_\varepsilon\subset\mathbb{R}^d$ such that $$\|f(\cdot+x)-f(\cdot)\|_{C^2(\mathbb{R}^d)}<\varepsilon, \quad \forall x\in\Lambda_\varepsilon.$$
	Define $O:=\{x+z_g\mid x\in\Lambda_\varepsilon,  g(\cdot)=f(\cdot+x)\}$. Let $R$ be the sum of the relatively dense distance of $\Lambda_\varepsilon$ and $2r$. Then, any ball of radius $R$ contains at least one point from $O$. For any $ x\in\Lambda_\varepsilon,  $ with $  g(\cdot)=f(\cdot+x)$, we have $$\nabla f(x+z_g)=\nabla g(z_g)=0,$$
	and $$\forall x_0,y_0\in \overline{B}_{r'}(x+z_g), \|\nabla f(x_0)-\nabla f(y_0)\|=\|\nabla g(x_0-x)-\nabla g(y_0-x)\|\geq m\|x_0-y_0\|.$$
	Hence $\nabla f$ satisfies the Aubry criterion with parameters $O,R,r',m$.
\end{proof}

\section{Hyperbolicity}\label{sec:hyperbolic}

Let $d>0$ be an integer. Consider a family of generating functions defined by\begin{align*}
    S_\lambda:\mathbb{R}^d\times\mathbb{R}^d&\rightarrow\mathbb{R}\\
    (x,y)&\mapsto I(x-y)+\lambda V(x)
\end{align*}
where $\lambda>0$, $I,V\in C^2(\mathbb{R}^d;\mathbb{R})$. Let $\mathcal{C}$ denote the set of all functions from $\mathbb{Z}$ to $\mathbb{R}^d$. A configuration $u\in\mathcal{C}$ is said to admit a rotation vector $\rho\in \operatorname{Hom}(\mathbb{Z},\mathbb{R}^d)$ if $$\lim_{i\to\pm\infty}\frac{u_i-\rho_i}{i}=0.$$
A sufficient condition for  $u$ to admit  $\rho$ as a rotation vector is $d(u,\rho)<+\infty.$
We define  $u\in\mathcal{C}$ to be an equilibrium configuration of $S_\lambda$ if it satisfies\begin{equation}\label{eq:equiconf}
    \nabla I(u_i-u_{i+1})-\nabla I(u_{i-1}-u_i)+\lambda\nabla V(u_i)=0, \quad \forall i\in\mathbb{Z}.
\end{equation}
Linearizing of \eqref{eq:equiconf} yields  the equation:\begin{equation}\label{eq:linearization}
    \nabla^2I(u_i-u_{i+1})(\xi_i-\xi_{i+1})-\nabla^2I(u_{i-1}-u_i)(\xi_{i-1}-\xi_i)+\lambda\nabla^2V(u_i)\xi_i=0,
\end{equation}
where $\xi_i\in T_{u_i}\mathbb{R}^d$ for all $i\in\mathbb{Z}$.
\begin{definition}\label{def:hyperbolic}
    We say that an equilibrium configuration $u\in\mathcal{C}$ is hyperbolic if there exists constants $\mu>1$, $\alpha\in (0,1)$, and $\beta\in (0,1)$ such that, for any $i\in\mathbb{Z}$, and for any $\xi_{i-1}\in T_{u_{i-1}}\mathbb{R}^d$, $\xi_{i}\in T_{u_{i}}\mathbb{R}^d$, $\xi_{i+1}\in T_{u_{i+1}}\mathbb{R}^d$ satisfying \eqref{eq:linearization}, the following conditions hold: \begin{enumerate}
        \item [(i)] if $\|\xi_{i-1}\|\leq \alpha\|\xi_i\|$, then $\|\xi_i\|\leq \alpha\|\xi_{i+1}\|$ and $\|(\xi_{i},\xi_{i+1})\|\geq \mu\|(\xi_{i-1},\xi_i)\|$;
        \item [(ii)] if $\|\xi_{i+1}\|\leq \beta\|\xi_i\|$, then $\|\xi_i\|\leq \beta\|\xi_{i-1}\|$ and $\|(\xi_{i-1},\xi_i)\|\geq\mu\|(\xi_{i},\xi_{i+1})\| $.
    \end{enumerate}
\end{definition}

\begin{remark}
	This definition differs from the usual definition of hyperbolicity, as it only assumes $I \in C^2(\mathbb{R}^d,\mathbb{R})$. However, under the assumption that  $I$  is uniformly strongly convex, we will demonstrate (see Theorem \ref{thm:DefofHyperbolic}) that this definition implies the standard hyperbolicity of the orbits of twist maps associated with  $S_\lambda$. 
\end{remark}

For now, we will prove that the equilibrium configurations of  $S_\lambda$ , as found via Theorem \ref{thm:main}, are all hyperbolic.
Define operators  $\Delta:\mathcal{C}\rightarrow\mathcal{C}$ and $\Psi:\mathcal{C}\rightarrow\mathcal{C}$ as follows:$$\Delta(u)_i=\nabla I(u_i-u_{i+1})-\nabla I(u_{i-1}-u_i), \quad \Psi(u)_i=\nabla V(u_i).$$ Then $u\in\mathcal{C}$ is an equilibrium configuration of $S_\lambda$ if and only if 
  $$  (\Delta+\lambda\Psi)u=0.$$
  
We will demonstrate later that $\Delta$ satisfies the following properties: it is invariant  with $\Delta|_{\operatorname{Hom}}=0$,  and it is Lipschitz continuous with a constant $K(\rho,R)$ on the closed ball $\{u\mid d(u,\rho)\leq R\}$, for all $\rho\in\operatorname{Hom}(\mathbb{Z},\mathbb{R}^d)$ and $R>0$ (see Lemma \ref{lemma:Delta}). If $\nabla V$ satisfies the Aubry criterion with parameters $O,R,r,m$, then Theorem \ref{thm:main} establishes that, for every  $\rho\in\operatorname{Hom}(\mathbb{Z},\mathbb{R}^d)$ and $\lambda\geq \frac{r+R}{rm}K(\rho,r+R)$, the generating function $S_\lambda$ admits an equilibrium configuration $u$ satisfying $\sup_{i\in\mathbb{Z}}\inf_{z\in O}\|u_i-z\|\leq r$ and  $d(u,\rho)\leq r+R$. 

\begin{theorem}\label{thm:hyperbolic}
	Let $I,V \in C^2(\mathbb{R}^d,\mathbb{R})$ and assume that $\nabla V$ satisfies the Aubry criterion with parameters $O,R,r,m$.
    If  $\rho\in\operatorname{Hom}(\mathbb{Z},\mathbb{R}^d)$, $\lambda\geq \frac{r+R}{rm}K(\rho,r+R)$, and $u$ is an equilibrium configuration satisfying $\sup_{i\in\mathbb{Z}}\inf_{z\in O}\|u_i-z\|\leq r$ and  $d(u,\rho)\leq r+R$, then $u$ is hyperbolic.
\end{theorem}

Before proving Theorem \ref{thm:hyperbolic}, we will first establish the following three lemmas. Throughout the discussion, let $\sigma_{\max}(A)$ (respectively, $\sigma_{\min}(A)$) denote the largest (respectively, smallest) singular value of a matrix $A$.

\begin{lemma}\label{lemma:Delta}
     The operator $\Delta$ satisfies the following properties:\begin{enumerate}
     	\item [(i)] $\Delta$ is invariant with $\Delta|_{\operatorname{Hom}}=0$;
     	\item [(ii)]  $\Delta$ is
        Lipschitz continuous on the closed ball  $\{u\mid d(u,\rho)\leq R\}$ for all $\rho\in\operatorname{Hom}(\mathbb{Z},\mathbb{R}^d)$ and $R>0$.
     \end{enumerate}
      The Lipschitz constant on $\{u\mid d(u,\rho)\leq R\}$ is given by 
      $$K(\rho,R):=4\sup_{x\in B_{\|\rho\|+2R}(0)}\sigma_{\max}(\nabla^2I(x)).$$
\end{lemma}

\begin{proof}
First, recall that the shift operator $S^k$ for $k\in\mathbb{Z}$ is defined by $S^k(u)_i=u_{i+k}$ for all $i\in\mathbb{Z}$. Then we have:
    \begin{align*}
        [\Delta\circ S^k(u)]_i&=\nabla I(S^k(u)_i-S^k(u)_{i+1})-\nabla I(S^k(u)_{i-1}-S^k(u)_i)\\ 
        &=\nabla I(u_{i+k}-u_{i+k+1})-\nabla I(u_{i+k-1}-u_{i+k})\\ 
        &=\Delta(u)_{i+k}\\ 
        &=[S^k\circ \Delta(u)]_i.
    \end{align*}
Next, recall that the translation operator $T^c$ along a vector $c\in\mathbb{R}^d$ is defined by $T^c(u)_i=u_i+c$ for all $i\in\mathbb{Z}$. Then:
\begin{align*}
    [\Delta\circ T^c(u)]_i&=\nabla I(T^c(u)_i-T^c(u)_{i+1})-\nabla I(T^c(u)_{i-1}-T^c(u)_i)\\ 
    &=\nabla I(u_i-u_{i+1})-\nabla I(u_{i-1}-u_i)\\ 
    &=\Delta(u)_i.
\end{align*}
To verify $\Delta|_{\operatorname{Hom}}=0$, let $\rho\in\operatorname{Hom}(\mathbb{Z},\mathbb{R}^d)$. Then:
\begin{align*}
    \Delta(\rho)_i
    &=\nabla I(\rho_i-\rho_{i+1})-\nabla I(\rho_{i-1}-\rho_{i})\\ 
    &=\nabla I(\rho_{-1})-\nabla I(\rho_{-1})\\ 
    &=0.
\end{align*}
For the Lipschitz continuity, consider $u,u'\in \{u\mid d(u,\rho)\leq R\}$ and calculate:
	\begin{align*}
    &\quad\,\|\Delta(u)_i-\Delta(u')_i\|\\ 
    &=\|\nabla I(u_i-u_{i+1})-\nabla I(u_{i-1}-u_i)-\nabla I(u'_i-u'_{i+1})+\nabla I(u'_{i-1}-u'_i)\|\\
    &\leq  \|\nabla I(u_i-u_{i+1})-\nabla I(u'_i-u'_{i+1})\|+\|\nabla I(u_{i-1}-u_i)-\nabla I(u'_{i-1}-u'_i)\|\\ 
    &\leq \sup_{x\in B_{\|\rho\|+2R}(0)}\sigma_{\max}(\nabla^2I(x))(\|u_i-u_{i+1}-u_i'+u_{i+1}'\|+\|u_{i-1}-u_{i}-u_{i-1}'+u_{i}'\|)\\ 
    &\leq 4\sup_{x\in B_{\|\rho\|+2R}(0)}\sigma_{\max}(\nabla^2I(x)) d(u,u').
  \end{align*}
  Thus $\Delta$ is Lipschitz continuous on  $\{u\mid d(u,\rho)\leq R\}$, with Lipschitz constant \[K(\rho,R)=4\sup_{x\in B_{\|\rho\|+2R}(0)}\sigma_{\max}(\nabla^2I(x)).\qedhere\]
\end{proof}

\begin{lemma}\label{lemma:Sigmamax}
    If $\rho\in\operatorname{Hom}(\mathbb{Z},\mathbb{R}^d)$ and $d(u,\rho)\leq r+R$, then $$\sigma_{\max}(\nabla^2I(u_i-u_{i+1}))\leq \frac{K(\rho,r+R)}{4}, \quad \forall i\in\mathbb{Z}.$$
\end{lemma}

\begin{proof}
Since $d(u,\rho)\leq r+R$, we have \begin{align*}
    \|u_i-u_{i+1}\|&\leq \|u_i-\rho_i\|+\|\rho_{i+1}-u_{i+1}\|+\|\rho\|\\ 
    &\leq 2(r+R)+\|\rho\|.
\end{align*}
Thus: \[
    \sigma_{\max}(\nabla^2I(u_i-u_{i+1}))\leq \sup_{x\in B_{\|\rho\|+2(r+R)}(0)}\sigma_{\max}(\nabla^2I(x))=\frac{K(\rho,r+R)}{4}.\qquad \qedhere
\]
\end{proof}

\begin{lemma}\label{lemma:Sigmamin}
    If $\nabla V$ satisfies the Aubry criterion with parameters $O,R,r,m$, and $u$ is a configuration satisfying $\sup_{i\in\mathbb{Z}}\inf_{z\in O}\|u_i-z\|\leq r$, then  $$\sigma_{\min}(\nabla^2V(u_i))\geq m>0, \quad \forall i\in\mathbb{Z}.$$
\end{lemma}

\begin{proof}
    For any $i\in\mathbb{Z}$, since $\sup_{i\in\mathbb{Z}}\inf_{z\in O}\|u_i-z\|\leq r$, there exists $z\in O$ such that  $u_i\in\overline{B}_r(z)$.
     For any $v\in\mathbb{R}^d$ with $\|v\|=1$, we compute \begin{align*}
        \|\nabla^2V(u_i)\cdot v\|=\lim_{t\to0}\left\|\frac{\nabla V(u_i+tv)-\nabla V(u_i)}{t}\right\|\geq  \lim_{t\to0}\frac{m\|tv\|}{|t|}=m.\qquad \qedhere
    \end{align*}
\end{proof}

\begin{proof}[Proof of Theorem \ref{thm:hyperbolic}]
Let $\mu$ and $\alpha$ denote the roots of the quadratic equation $x^2 - (2+\frac{4R}{r})x + 1 = 0$, where $\mu >1$ and $\alpha<1$. We first verify item (i) of Definition \ref{def:hyperbolic}.
    Assume $\|\xi_{i-1}\|\leq \alpha\|\xi_i\|$. Then, we have \begin{align*}
        &\frac{K(\rho,r+R)}{4}\|\xi_{i+1}\|\\ 
        \geq &\|\nabla^2I(u_i-u_{i+1})\xi_{i+1}\|\\ 
        = &\| \nabla^2I(u_i-u_{i+1})\xi_i-\nabla^2I(u_{i-1}-u_i)(\xi_{i-1}-\xi_i)+\lambda\nabla^2V(u_i)\xi_i\|\\ 
        \geq &\left(\lambda\sigma_{\min}(\nabla^2V(u_i))-\frac{K(\rho,r+R)}{4}-\frac{K(\rho,r+R)}{4}-\frac{ K(\rho,r+R)}{4}\alpha\right)\|\xi_i\|\\ 
        \geq &\frac{K(\rho,r+R)}{4}\left(2+\frac{4R}{r}-\alpha\right)\|\xi_i\|\\ 
        =&\frac{K(\rho,r+R)}{4}\cdot\frac{1}{\alpha}\|\xi_i\|,
    \end{align*}
    where the three inequalities follow from Lemma \ref{lemma:Sigmamax}, the triangle inequality, and Lemma \ref{lemma:Sigmamin}, respectively.
    
    Since $\|\xi_{i-1}\|\leq \alpha\|\xi_i\|$ and $\|\xi_{i}\|\leq \alpha\|\xi_{i+1}\|$, it follows that\begin{align*}
        \|(\xi_i,\xi_{i+1})\|&=\sqrt{\|\xi_i\|^2+\|\xi_{i+1}\|^2}\\ 
        &\geq \sqrt{\frac{1}{\alpha^2}\|\xi_{i-1}\|^2+\frac{1}{\alpha^2}\|\xi_{i}\|^2}\\ 
        &=\frac{1}{\alpha}\|(\xi_{i-1},\xi_i)\|\\ 
        &=\mu\|(\xi_{i-1},\xi_i)\|.
    \end{align*}
    The proof of item (ii) in Definition \ref{def:hyperbolic} follows a similar argument.
    \end{proof}

Let $I:\mathbb{R}^d\rightarrow\mathbb{R}$ be a $C^2$ function such that $$0<\varepsilon\leq \sigma_{\min}(\nabla^2I(x))\leq \sigma_{\max}(\nabla^2I(x))\leq\mathcal{E}<+\infty,\quad \forall x\in\mathbb{R}^d$$
for some constants $\varepsilon$ and $\mathcal{E}$. The classical results from \cite{golé2001symplectic,katok1995introduction} establish the following: \begin{enumerate}
	\item [(i)] $S_\lambda$ can be associated with the twist map $$F_\lambda:\mathbb{R}^d\times\mathbb{R}^d\rightarrow\mathbb{R}^d\times\mathbb{R}^d$$
 defined by: \begin{align*}
    F_\lambda(x,p)=(y,p') \Longleftrightarrow\begin{cases}
      p=-\nabla I(x-y)-\lambda \nabla V(x)\\ 
      p'=-\nabla I(x-y)
    \end{cases}.
  \end{align*}

    \item [(ii)]$F_\lambda$ is a $C^1$-diffeomorphism on $\mathbb{R}^d\times\mathbb{R}^d$;
    \item [(iii)]$u\in\mathcal{C}$ is an equilibrium configuration of $S_\lambda$ if and only if  $F_\lambda(u_i,p_i)=(u_{i+1},p_{i+1}) $ for all $i\in\mathbb{Z}$, where $p_i=-\nabla I(u_i-u_{i+1})-\lambda \nabla V(u_i)$.\end{enumerate}

\begin{theorem}\label{thm:DefofHyperbolic}
	Let $\lambda>0$ and $I,V\in C^2(\mathbb{R}^d;\mathbb{R})$. Suppose there exist $\varepsilon>0$ and $\mathcal{E}>0$ such that  $$0<\varepsilon\leq \sigma_{\min}(\nabla^2I(x))\leq \sigma_{\max}(\nabla^2I(x))\leq\mathcal{E}<+\infty,\quad \forall x\in\mathbb{R}^d.$$
	If $u$ is a hyperbolic equilibrium configuration of $S_\lambda$, then the sequence $\{(u_i,p_i)\}_{i\in\mathbb{Z}}$ is a hyperbolic orbit of $F_\lambda$, where  $p_i=-\nabla I(u_i-u_{i+1})-\lambda \nabla V(u_i)$ and $F_\lambda$ is the twist map associated with $S_\lambda$.
\end{theorem}

\begin{proof}
	Consider the discrete Legendre transform defined as: \begin{align*}
		L:\mathbb{R}^d\times\mathbb{R}^d&\rightarrow\mathbb{R}^d\times\mathbb{R}^d\\ 
		(x,y)&\mapsto (y,-\nabla I(x-y)).
	\end{align*}
	Since  $I$ is a $C^2$ function satisfying $$0<\varepsilon\leq \sigma_{\min}(\nabla^2I(x))\leq \sigma_{\max}(\nabla^2I(x))\leq\mathcal{E}<+\infty,\quad \forall x\in\mathbb{R}^d.$$
	We can estimate the upper bound of largest singular value of $$DL(x,y)=\begin{bmatrix}
		\mathbb{O} &\mathbb{I}\\ 
		-\nabla^2I(x-y)& \nabla^2 I(x-y)
	\end{bmatrix}$$
	as follows: \begin{align*}
		\sigma_{\max}(DL(x,y))=&\sup_{\xi_1,\xi_2\in\mathbb{R}^d,\|\xi_1\|^2+\|\xi_2\|^2\leq 1}\left\|DL(x,y)(
			\xi_1,
			\xi_2)
		\right\|\\ 
	= & \sup_{\xi_1,\xi_2\in\mathbb{R}^d,\|\xi_1\|^2+\|\xi_2\|^2\leq 1}\left\|
			(\xi_2,			\nabla^2 I(x-y)(\xi_2-\xi_1))
		\right\|\\ 
	= & \sup_{\xi_1,\xi_2\in\mathbb{R}^d,\|\xi_1\|^2+\|\xi_2\|^2\leq 1}\sqrt{\|\xi_2\|^2+\|\nabla^2 I(x-y)(\xi_2-\xi_1)\|^2}\\ 
	\leq  &\sup_{\xi_1,\xi_2\in\mathbb{R}^d,\|\xi_1\|^2+\|\xi_2\|^2\leq 1}\sqrt{\|\xi_2\|^2+\mathcal{E}^2\|\xi_2-\xi_1\|^2}\\ 
	\leq  & \sqrt{1+4\mathcal{E}^2}.
		\end{align*}
		Similarly, the upper bound of the largest singular value of $DL^{-1}(x,y)$  can be estimated as:$$\sigma_{\max}(DL^{-1}(x,y))\leq \sqrt{1+4\varepsilon^{-2}}. $$
	Define $\Sigma_\lambda:=L^{-1}\circ F_\lambda\circ L$. Then $u$ is an equilibrium configuration if and only if $\Sigma_\lambda(u_{i-1},u_i)=(u_i,u_{i+1})$ for all $i\in\mathbb{Z}$. Consequently, $\{\xi_i\}_{i\in\mathbb{Z}}$ satisfies \eqref{eq:linearization} if and only if $$D\Sigma_\lambda(u_{i-1},u_i)(\xi_{i-1},\xi_i)=(\xi_{i},\xi_{i+1}),\quad \forall i\in\mathbb{Z}.$$
	If  $u$ is a hyperbolic equilibrium configuration, there exists $\mu>1$, $\alpha\in(0,1)$, and $\beta\in (0,1)$ such that Definition \ref{def:hyperbolic} holds.
	Define two cones in $T_{u_{i-1}}\mathbb{R}^d\oplus T_{u_{i}}\mathbb{R}^d$ as: \begin{align*}
		C^u_i:=\{(\xi_{i-1},\xi_{i})\mid \|\xi_{i-1}\|\leq \alpha\|\xi_{i}\|\}\quad \text{and }\quad C^s_i:=\{(\xi_{i-1},\xi_{i})\mid \|\xi_{i}\|\leq \beta\|\xi_{i-1}\|\}.
	\end{align*}
	 Then: \begin{enumerate}
	 	\item [(i)] $D\Sigma_\lambda(u_{i-1},u_i)(C^u_i)\subset C^u_{i+1}$;
	 	\item [(ii)]$D\Sigma_\lambda^{-1}(u_{i},u_{i+1})(C^s_{i+1})\subset C^s_{i}$;
	 	\item [(iii)]$\|D\Sigma_\lambda(u_{i-1},u_i)(\xi_{i-1},\xi_i)\|\geq \mu\|(\xi_{i-1},\xi_i)\|, \forall (\xi_{i-1},\xi_i)\in C^u_i$;
	 	\item [(iv)]$\|D\Sigma_\lambda^{-1}(u_{i},u_{i+1})(\xi_{i},\xi_{i+1})\|\geq \mu\|(\xi_{i},\xi_{i+1})\|, \forall (\xi_{i},\xi_{i+1})\in C^s_{i+1}$.
	 \end{enumerate}
	 Using the Alekseev Cone Field Criterion \cite[Chapter 6]{katok1995introduction} for $\Sigma_\lambda$, we obtain a hyperbolic decomposition: $\{T_{u_{i-1}}\mathbb{R}^d\oplus T_{u_{i}}\mathbb{R}^d=E^u_i\oplus E^s_i\}_{i\in\mathbb{Z}}$ along the orbit $\{(u_{i-1},u_i)\}_{i\in\mathbb{Z}}$ of $\Sigma_\lambda$. Consequently, for $F_\lambda$, we get:$$T_{u_i}\mathbb{R}^d\oplus T_{p_{i}}\mathbb{R}^d=(DL(u_{i-1},u_i)E_i^u)\oplus(DL(u_{i-1},u_i)E_i^s), \quad \forall i\in\mathbb{Z}.$$
	 Since $F_\lambda=L\circ \Sigma_\lambda\circ L^{-1}$, this decomposition is $DF_\lambda$-invariant, and for any  $n\geq 0$, we have: \begin{align*}
	 	\|DF_\lambda^n(u_i,p_i)(\xi_i,\zeta_i)\|=&\|DL\circ D\Sigma_\lambda^n\circ DL^{-1}(u_i,p_i)(\xi_i,\zeta_i)\|\\ 
	 	\leq &\sqrt{1+4\varepsilon^{-2}}\sqrt{1+4\mathcal{E}^2}\mu^{-n}\|(\xi_i,\zeta_i)\|, \quad \forall (\xi_i,\zeta_i)\in DL(u_{i-1},u_i)E_i^s
	 \end{align*}
	 and \begin{align*}
	 	\|DF_\lambda^{-n}(u_i,p_i)(\xi_i,\zeta_i)\|=&\|DL^{-1}\circ D\Sigma_\lambda^{-n}\circ DL(u_i,p_i)(\xi_i,\zeta_i)\|\\ 
	 	\leq &\sqrt{1+4\varepsilon^{-2}}\sqrt{1+4\mathcal{E}^2}\mu^{-n}\|(\xi_i,\zeta_i)\|, \quad \forall (\xi_i,\zeta_i)\in DL(u_{i-1},u_i)E_i^u.
	 \end{align*}
	 
	 For any $i\in\mathbb{Z}$, let   $v_i\in DL(u_{i-1},u_i)E_i^u$ and $w_i\in DL(u_{i-1},u_i)E_i^s$ be two unit vectors. Then, we have the inequality:\begin{align*}
	 	\|v_i-w_i\| 
	 	&\geq \frac{1}{\sqrt{1+4\varepsilon^{-2}}}\|DL^{-1}(u_{i-1},u_i)v_i-DL^{-1}(u_{i-1},u_i)w_i\|.	 \end{align*}
	 	Define  $v_i':=DL^{-1}(u_{i-1},u_i)v_i$ and $w_i':=DL^{-1}(u_{i-1},u_i)w_i$. Clearly, $v_i'\in E_i^u$ and $w_i'\in E_i^s$.
	 Given that $\alpha\in(0,1)$ and $\beta\in (0,1)$, there exists a uniform angle $\gamma>0$ between the two cone fields $\{C_i^u\}_{i\in\mathbb{Z}}$ and $\{C_i^s\}_{i\in\mathbb{Z}}$. Consequently,  $\gamma$ is also the uniform angle between $v_i'$ and $w_i'$. Additionally, since $\|v_i'\|\geq \frac{1}{\sqrt{1+4\mathcal{E}^2}}$ and $\|w_i'\|\geq \frac{1}{\sqrt{1+4\mathcal{E}^2}}$, the angle $\gamma$ ensures the existence of a constant $\delta>0$, independent of $i\in\mathbb{Z}$, such that: $$\|v_i'-w_i'\|\geq \delta.$$ This implies: $$\|v_i-w_i\|\geq \frac{\delta}{\sqrt{1+4\varepsilon^{-2}}}>0.$$ Thus, there exists a uniform angle between $\{DL(u_{i-1},u_i)E_i^u\}_{i\in\mathbb{Z}}$ and $\{DL(u_{i-1},u_i)E_i^s\}_{i\in\mathbb{Z}}$.
	 
	 	As a result,  $\{(u_i,p_i)\}_{i\in\mathbb{Z}}$ is a hyperbolic orbit of $F_\lambda$.
	 \end{proof}

\section*{Acknowledgments}
The authors thank Yujia An and Dongyu Yao for a careful reading of the manuscripts and many helpful suggestions.

\section*{Declaration of interests}
The authors declare that they have no known competing financial interests or personal relationships that could have appeared to influence the work reported in this paper.

\bibliographystyle{is-abbrv}  
\bibliography{references}  

\begin{thebibliography}{10}
\ifx \showCODEN  \undefined \def \showCODEN #1{CODEN #1}  \fi
\ifx \showISBN   \undefined \def \showISBN  #1{ISBN #1}   \fi
\ifx \showISSN   \undefined \def \showISSN  #1{ISSN #1}   \fi
\ifx \showLCCN   \undefined \def \showLCCN  #1{LCCN #1}   \fi
\ifx \showPRICE  \undefined \def \showPRICE #1{#1}        \fi
\ifx \showURL    \undefined \def \showURL {URL }          \fi
\ifx \path       \undefined \input path.sty               \fi
\ifx \ifshowURL \undefined
     \newif \ifshowURL
     \showURLtrue
\fi

\bibitem{ADSWY}
Y.~An, R.~de~la Llave, X.~Su, D.~Wang, and D.~Yao.
\newblock {KAM} theory for almost-periodic equilibria in one dimensional
  almost-periodic media.
\newblock {\em Preprint at https://arxiv.org/abs/2411.05367}, 2024.

\bibitem{Aubry83}
S.~Aubry.
\newblock The twist map, the extended {F}renkel-{K}ontorova model and the
  devil's staircase.
\newblock volume~7, pages 240--258. 1983.
\newblock \showISSN{0167-2789,1872-8022}.
\newblock \ifshowURL {\showURL
  \path|https://doi.org/10.1016/0167-2789(83)90129-X|}\fi.
\newblock Order in chaos (Los Alamos, N.M., 1982).

\bibitem{Aubry1990}
S.~Aubry and G.~Abramovici.
\newblock Chaotic trajectories in the standard map. {T}he concept of
  anti-integrability.
\newblock {\em Phys. D}, 43\penalty0 (2-3):\penalty0 199--219, 1990.
\newblock \showISSN{0167-2789}.
\newblock \ifshowURL {\showURL
  \path|https://doi.org/10.1016/0167-2789(90)90133-A|}\fi.

\bibitem{Bolotin_2015}
S.~V. Bolotin and D.~V. Treschev.
\newblock The anti-integrable limit.
\newblock {\em Russian Mathematical Surveys}, 70\penalty0 (6):\penalty0 975,
  dec 2015.
\newblock \ifshowURL {\showURL
  \path|https://dx.doi.org/10.1070/RM2015v070n06ABEH004972|}\fi.

\bibitem{BK2004}
O.~M. Braun and Y.~S. Kivshar.
\newblock {\em The {F}renkel-{K}ontorova model}.
\newblock Texts and Monographs in Physics. Springer-Verlag, Berlin, 2004.
\newblock \showISBN{3-540-40771-5}.
\newblock xviii+472 pp.
\newblock \ifshowURL {\showURL
  \path|http://dx.doi.org.prx.library.gatech.edu/10.1007/978-3-662-10331-9|}\fi.

\bibitem{Rafaelbook}
R.~de~la Llave.
\newblock A tutorial on {KAM} theory.
\newblock In {\em Smooth ergodic theory and its applications ({S}eattle, {WA},
  1999)}, volume~69 of {\em Proc. Sympos. Pure Math.}, pages 175--292. Amer.
  Math. Soc., Providence, RI, 2001.
\newblock \showISBN{0-8218-2682-4}.
\newblock \ifshowURL {\showURL
  \path|https://doi.org/10.1090/pspum/069/1858536|}\fi.

\bibitem{dusu2021}
J.~Du and X.~Su.
\newblock On the existence of solutions for the {F}renkel-{K}ontorova models on
  quasi-crystals.
\newblock {\em Electronic Research Archive}, 29\penalty0 (6):\penalty0
  4177--4198, 2021.
\newblock \showISSN{2688-1594}.
\newblock \ifshowURL {\showURL
  \path|https://www.aimspress.com/article/doi/10.3934/era.2021078|}\fi.

\bibitem{FK39}
J.~Frenkel and T.~Kontorova.
\newblock On the theory of plastic deformation and twinning.
\newblock {\em Acad. Sci. URSS, J. Physics}, 1:\penalty0 137--149, 1939.

\bibitem{GGP2006}
J.-M. Gambaudo, P.~Guiraud, and S.~Petite.
\newblock Minimal configurations for the {F}renkel-{K}ontorova model on a
  quasicrystal.
\newblock {\em Communications in Mathematical Physics}, 265\penalty0
  (1):\penalty0 165--188, 2006.

\bibitem{GPT2017}
E.~Garibaldi, S.~Petite, and P.~Thieullen.
\newblock Calibrated configurations for {F}renkel-{K}ontorova type models in
  almost periodic environments.
\newblock {\em Ann. Henri Poincar\'{e}}, 18\penalty0 (9):\penalty0 2905--2943,
  2017.
\newblock \showISSN{1424-0637}.
\newblock \ifshowURL {\showURL
  \path|https://doi.org/10.1007/s00023-017-0589-7|}\fi.

\bibitem{golé2001symplectic}
C.~Gol{\'e}.
\newblock {\em Symplectic Twist Maps: Global Variational Techniques}.
\newblock Advanced series in nonlinear dynamics. World Scientific, 2001.
\newblock \showISBN{9789812810762}.
\newblock \ifshowURL {\showURL
  \path|https://books.google.com/books?id=qhni_1MrvkQC|}\fi.

\bibitem{golubitsky1974stable}
M.~Golubitsky and V.~Guillemin.
\newblock {\em Stable Mappings and Their Singularities}.
\newblock Graduate texts in mathematics. Springer, 1974.
\newblock \showISBN{9787506200448}.
\newblock \showLCCN{73018276}.
\newblock \ifshowURL {\showURL
  \path|https://books.google.co.jp/books?id=DhQJcgAACAAJ|}\fi.

\bibitem{Goroff}
D.~L. Goroff.
\newblock Hyperbolic sets for twist maps.
\newblock {\em Ergodic Theory Dynam. Systems}, 5\penalty0 (3):\penalty0
  337--339, 1985.
\newblock \showISSN{0143-3857,1469-4417}.
\newblock \ifshowURL {\showURL
  \path|https://doi.org/10.1017/S0143385700002996|}\fi.

\bibitem{katok1995introduction}
A.~Katok, A.~Katok, and B.~Hasselblatt.
\newblock {\em Introduction to the Modern Theory of Dynamical Systems}.
\newblock Encyclopedia of Mathematics and its Applications. Cambridge
  University Press, 1995.
\newblock \showISBN{9780521575577}.
\newblock \showLCCN{94026547}.
\newblock \ifshowURL {\showURL
  \path|https://books.google.ca/books?id=9nL7ZX8Djp4C|}\fi.

\bibitem{Mackay_1992}
R.~S. Mackay and J.~D. Meiss.
\newblock Cantori for symplectic maps near the anti-integrable limit.
\newblock {\em Nonlinearity}, 5\penalty0 (1):\penalty0 149, jan 1992.
\newblock \ifshowURL {\showURL
  \path|https://dx.doi.org/10.1088/0951-7715/5/1/006|}\fi.

\bibitem{Mather82}
J.~N. Mather.
\newblock Existence of quasiperiodic orbits for twist homeomorphisms of the
  annulus.
\newblock {\em Topology}, 21\penalty0 (4):\penalty0 457--467, 1982.
\newblock \showISSN{0040-9383}.
\newblock \ifshowURL {\showURL
  \path|https://doi.org/10.1016/0040-9383(82)90023-4|}\fi.

\bibitem{munkres2000topology}
J.~Munkres.
\newblock {\em Topology}.
\newblock Featured Titles for Topology. Prentice Hall, Incorporated, 2000.
\newblock \showISBN{9780131816299}.
\newblock \showLCCN{99052942}.
\newblock \ifshowURL {\showURL
  \path|https://books.google.co.jp/books?id=XjoZAQAAIAAJ|}\fi.

\bibitem{SuL2012b}
X.~Su and R.~de~la Llave.
\newblock K{AM} theory for quasi-periodic equilibria in 1{D} quasi-periodic
  media: {II}. {L}ong-range interactions.
\newblock {\em J. Phys. A}, 45\penalty0 (45):\penalty0 455203, 24, 2012.
\newblock \showISSN{1751-8113,1751-8121}.
\newblock \ifshowURL {\showURL
  \path|https://doi.org/10.1088/1751-8113/45/45/455203|}\fi.

\bibitem{SuL2012}
X.~Su and R.~de~la Llave.
\newblock K{AM} {T}heory for {Q}uasi-periodic {E}quilibria in
  {O}ne-{D}imensional {Q}uasi-periodic {M}edia.
\newblock {\em SIAM J. Math. Anal.}, 44\penalty0 (6):\penalty0 3901--3927,
  2012.
\newblock \showCODEN{SJMAAH}.
\newblock \showISSN{0036-1410}.
\newblock \ifshowURL {\showURL \path|http://dx.doi.org/10.1137/12087160X|}\fi.

\bibitem{Trevino2019}
R.~Trevi\~{n}o.
\newblock Equilibrium configurations for generalized {F}renkel-{K}ontorova
  models on quasicrystals.
\newblock {\em Comm. Math. Phys.}, 371\penalty0 (1):\penalty0 1--17, 2019.
\newblock \showISSN{0010-3616}.
\newblock \ifshowURL {\showURL
  \path|https://doi.org/10.1007/s00220-019-03557-7|}\fi.

\end{thebibliography}

\end{document}